\documentclass[12pt,fleqn]{amsart}

\usepackage{amsmath}
\usepackage{amsthm}
\usepackage{amssymb}
\usepackage{tikz}
\usetikzlibrary{patterns}
\usepackage{pgfplots}
\usepackage{mathrsfs}
\usepackage[colorlinks=false]{hyperref}
\usepackage{paralist}
\usepackage[paperwidth=210mm,paperheight=297mm,inner=3.5cm,outer=3.5cm,top=2.5cm,bottom=2.6cm]{geometry}

\theoremstyle{plain}
\newtheorem{Theorem}{Thm}[section]
\newtheorem{Thm}[Theorem]{Theorem}
\newtheorem{Lem}[Theorem]{Lemma}

\newtheorem{Prop}[Theorem]{Proposition}
\newtheorem*{Thm*}{Theorem}
\newtheorem*{Prop*}{Proposition}

\theoremstyle{definition}
\newtheorem{Def}[Theorem]{Definition}
\newtheorem{Rem}[Theorem]{Remark}
\newtheorem{Exm}[Theorem]{Example}

\newtheorem{Const}[Theorem]{Constraint}

\newcommand\mbb{\mathbb}

\newcommand\mcal{\mathcal}

\newcommand\A{\mbb{A}}
\newcommand\C{\mbb{C}}
\newcommand\N{\mbb{N}}
\renewcommand\P{\mbb{P}}
\newcommand\Q{\mbb{Q}}
\newcommand\PP{\mbb{P}}
\newcommand\R{\mbb{R}}
\newcommand\Z{\mbb{Z}}
\newcommand\V{\mcal{V}}
\newcommand\ul{\underline}
\DeclareMathOperator\hilb{Hilb}
\DeclareMathOperator\gor{Gor}
\DeclareMathOperator\rk{rk}
\DeclareMathOperator\crk{crk}
\DeclareMathOperator\cl{cl}
\DeclareMathOperator\lspan{span}
\DeclareMathOperator\ev{ev}
\DeclareMathOperator\divi{Div}

\DeclareMathOperator\exr{Exr}
\newcommand\I{\mathcal{I}}
\newcommand\ratto{\dashrightarrow}

\begin{document}
\title{Extreme Rays of Hankel Spectrahedra for Ternary Forms}
\author{Grigoriy Blekherman}
\address{Georgia Institute of Technology, Atlanta, Georgia, USA}
\email{greg@math.gatech.edu}
\author{Rainer Sinn}
\address{Georgia Institute of Technology, Atlanta, Georgia, USA}
\email{sinn@math.gatech.edu}

\begin{abstract}
The cone of sums of squares is one of the central objects in convex algebraic geometry. Its defining linear inequalities correspond to the extreme rays of the dual convex cone. This dual cone is a spectrahedron, which can be explicitly realized as a section of the cone of positive semidefinite matrices with the linear subspace of Hankel (or middle catalecticant) matrices. In this paper we initiate a systematic study of the extreme rays of Hankel spectrahedra for ternary forms. 
We show that the Zariski closure of the union of extreme rays is the variety of all Hankel matrices of corank at least $4$, an irreducible variety of codimension $10$ and we determine its degree. We explicitly construct an extreme ray of maximal rank using the Cayley-Bacharach Theorem for plane curves. 
We apply our results to the study of the algebraic boundary of the cone of sums of squares. Its irreducible components are dual varieties to varieties of Gorenstein ideals with certain Hilbert functions. We determine these Hilbert functions for some cases of small degree. We also observe surprising gaps in the ranks of Hankel matrices of the extreme rays.
\end{abstract}

\maketitle

\section*{Introduction}
The following convex cones are fundamental objects in convex algebraic geometry: the cone $P_{n,2d}$ of homogeneous polynomials (forms) of degree $2d$ in $\R[x_1,\ldots,x_n]$ that are nonnegative on $\R^n$, and the cone $\Sigma_{n,2d}$ consisting of sums of squares of degree $2d$. Hilbert showed that only in the following three cases every nonnegative form is a sum of squares of forms: bivariate forms, quadratic forms, and ternary forms of degree 4. In all other cases Hilbert showed the existence of nonnegative polynomials that are not sums of squares  \cite{HilMR1510517}. 

The dual cones $P_{n,2d}^\vee$ and $\Sigma_{n,2d}^\vee$ consist of all linear functionals nonnegative on the corresponding primal cone. The \textit{extreme rays} of the dual cones provide the defining linear inequalities of the primal cones. Therefore, understanding extreme rays of $\Sigma^\vee_{n,2d}$ is crucial in understanding the boundary of the cone $\Sigma_{n,2d}$, as well as the difference between the cones $P_{n,2d}$ and $\Sigma_{n,2d}$. In the cases where there exist nonnegative polynomials that are not sums of squares, $\Sigma_{n,2d}^\vee$ must contain extreme rays that do not belong to $P^\vee_{n,2d}$. In recent years there has been considerable progress in understanding the extreme rays of $\Sigma_{n,2d}^\vee$ and the \emph{algebraic boundary} of $\Sigma_{n,2d}$, i.e.~the Zariski closure of its Euclidean boundary, in the two smallest cases where nonnegative polynomials are not equal to sums of squares: $n=3,$ $2d=6$ and $n=4,$ $2d=4$ \cite{BleNP, BleHauOttRanStuMR2999301}. In \cite{BleNP}, extreme rays of $\Sigma_{3,6}^\vee$ and $\Sigma_{4,4}^\vee$ were described using the Cayley-Bacharach theorem. In \cite{BleHauOttRanStuMR2999301}, this description led to a quite surprising connection between the algebraic boundaries of $\Sigma_{3,6}$ and $\Sigma_{4,4}$ and moduli spaces of K3 surfaces. In \cite{BlePositiveGorensteinIdeals}, the first author related the study of extreme rays of $\Sigma_{n,2d}^\vee$ to the associated \textit{Gorenstein ideals}.

Taking these results as a point of departure, we begin a systematic study of extreme rays of the cone $\Sigma_{n,2d}^\vee$ for ternary forms, i.e. $n=3$. We will denote the associated cones simply by $\Sigma_{2d}$ and $\Sigma_{2d}^\vee$. Our main technical tool will be the Buchsbaum-Eisenbud structure theorem for ternary Gorenstein ideals, and its refined analysis by Diesel in \cite{Die}. We will see that irreducible components of the algebraic boundary of $\Sigma_{2d}$ are dual varieties to varieties of Gorenstein ideals with certain Hilbert functions. This gives us a beautiful melding of convex geometry, commutative algebra, and algebraic geometry. 

The case of $2d=6$ was completely described in \cite{BleNP, BleHauOttRanStuMR2999301} and therefore we restrict our attention to $2d\geq 8$. Our first main result deals with the Zariski closure of the set of all extreme rays of $\Sigma_{2d}^\vee$ and tells us that extreme rays of $\Sigma_{2d}^\vee$ are plentiful, when compared to extreme rays of $P_{2d}^\vee$.

\begin{Thm*}[Theorem {\ref{Thm:ZarClExtRays}}]
For any $d\geq 4$, the Zariski closure of the set of extreme rays of $\Sigma_{2d}^\vee$ is the variety of Hankel matrices of corank at least $4$. It is irreducible, has codimension $10$, and degree $\prod_{\alpha = 0}^{3} \binom{N+\alpha}{4-\alpha}/\binom{2\alpha +1}{\alpha}$, where $N = \binom{d+2}{2}$.
\end{Thm*}
By contrast, the Zariski closure of the extreme rays of $P_{2d}^\vee$ is the $2d$-th Veronese embedding of $\PP^2$ and has dimension $2$ \cite[Chaper 4]{BleParThoMR3075433}. Note that for $\Sigma_6^\vee$, it follows from results of \cite{BleNP, BleHauOttRanStuMR2999301} that the Zariski closure of the set of extreme rays is the variety of Hankel matrices of corank at least $3$. It has dimension $21$, codimension $6$ and degree $2640$. Existence of extreme rays of co-rank $4$ is shown via an intricate explicit construction, which makes heavy use of Cayley-Bacharach theorem for plane curves. The details are given in Section \ref{sec:ExtRays}. 

The extreme rays of the dual cone $\Sigma_{2d}^\vee$ are stratified by the rank of the associated Hankel (middle catalecticant) matrix. 
This intricate stratification characterizes the algebraic boundary of the sums of squares cone via projective duality theory. We show the following theorem in section \ref{sec:ExtRays}.
%Let $\gor(T)$ denote the quasi-projective variety of all Gorenstein ideals with Hilbert function $T$, and let $\partial_a$ denote the algebraic boundary. We show the following:

\begin{Thm*}[Theorem \ref{Thm:algboundandgors}]
Let $X$ be an irreducible component of the algebraic boundary of $\Sigma_{2d}$. Then its dual projective variety $X^\ast$ is a subvariety of the Zariski closure of the union of extreme rays of $\Sigma_{2d}^\vee$, i.e.~the variety of Hankel matrices of corank $\geq 4$. Moreover, there is a Hilbert function $T$ such that the quasiprojective variety $\gor(T)$ of all Gorenstein ideals with Hilbert function $T$ is Zariski dense in $X^\ast$. 
\end{Thm*}

We work out the first three nontrivial cases $d=3,4,5$ in Section \ref{sec:Dextics}, extending the study of the algebraic boundary of the sums of squares cones for ternary sextics and quaternary quartics in \cite{BleHauOttRanStuMR2999301}.
More specifically we show in section \ref{sec:Dextics}: %\gb{I am pretty sure that at least for degree 8 these are the only Hilbert functions. Is that worth spending some extra time on? I don't remember how complicated the situation is in degree 10}:

\begin{Prop*}
The Hankel spectrahedron $\Sigma_8^\vee$ has extreme rays of rank $1$, $10$, and $11$. We construct extreme rays of rank $10$ and $11$ such that the Hilbert function of the corresponding Gorenstein ideal is
\begin{eqnarray*}
T_{10} & = & (1,3,6,9,10,9,6,3,1) \text{ and} \\
T_{11} & = & (1,3,6,10,11,10,6,3,1),
\end{eqnarray*}
respectively. %Let $\gor(T_r)$ denote the quasi-projective variety of all Gorenstein ideals with Hilbert function $T_r$. 
The dual varieties to $\gor(T_{10})$ and $\gor(T_{11})$ are irreducible components of the algebraic boundary of $\Sigma_8$.
\end{Prop*}

It is possible to show using refined analysis of \cite[Proposition 3.9]{Die} that these are the only Hilbert functions of Gorenstein ideals corresponding to extreme rays of $\Sigma_8^\vee$, and thus the algebraic boundary of $\Sigma_8$ has $3$ irreducible components: the discriminant, which is dual to rank $1$ extreme rays, and the dual varieties to $\gor(T_{10})$ and $\gor(T_{11})$.

\begin{Thm*}[Theorem {\ref{Thm:d5}} for $d=5$]
For every $r\in\{13,\ldots,17\}$, there is an extreme ray $\R_+\ell_r$ of $\Sigma_{10}^\vee$ such that the rank of the Hankel matrix $B_{\ell_r}$ is $r$. We construct extreme rays such that the Hilbert function $T_r$ of the associated Gorenstein ideal $I(\ell_r)$ is:
\begin{eqnarray*}
T_{13} & = & (1,3,6,9,12,13,12,9,6,3,1), \\
T_{14} & = & (1,3,6,10,13,14,13,10,6,3,1), \\
T_{15} & = & (1,3,6,10,14,15,14,10,6,3,1), \\
T_{16} & = & (1,3,6,10,14,16,14,10,6,3,1), \\
T_{17} & = & (1,3,6,10,15,17,15,10,6,3,1). 
\end{eqnarray*}
%Let $\gor(T_r)$ denote the quasi-projective variety of all Gorenstein ideals with Hilbert function $T_r$. 
The dual varieties to $\gor(T_r)$ form irreducible components of the algebraic boundary of the sums of squares cone $\Sigma_{10}$ for all $r\in\{13,\ldots,17\}$.
\end{Thm*}

It follows from the above the theorem that the algebraic boundary of $\Sigma_{10}$ has at least $6$ irreducible components. We conjecture that this list is complete. 

We saw previously that for $d \geq 4$ the minimal co-rank of an extreme ray is $4$. It was shown in \cite{BlePositiveGorensteinIdeals} that $\Sigma_{2d}^\vee$ has no extreme rays of rank $r$ with $1<r<3d-2$. We also see form the above results that this is the only gap in rank of extreme rays for $\Sigma_{8}^\vee$ and $\Sigma_{10}^\vee$. Surprisingly, the cone $\Sigma_{12}^{\vee}$ has another gap in possible ranks of extreme rays. We show the following theorem in section \ref{sec:Dextics}. %\gb{switch the order}

\begin{Thm*}
The cones $\Sigma_{2d}^\vee$ for $d=4,5$ have extreme rays of rank $r$ for $r=1$ and all $r$ such that $3d-2\leq r \leq \binom{d+2}{2}-4$. The cone $\Sigma_{12}^{\vee}$ has no extreme ray of rank $17$, but has extreme rays of rank $r$ for all $16\leq r \leq24$, $r\neq 17$. 
\end{Thm*}

We leave the reader with the following open questions:\\
\textbf{Open Questions:}
%\begin{oq}
\textit{
\begin{enumerate}
\item What are the possible ranks of Hankel matrices of extreme rays of $\Sigma_{2d}^\vee$?
\item Given the rank of a Hankel matrix of an extreme ray of $\Sigma_{2d}^\vee$, what are the possible Hilbert functions of the associated Gorenstein ideal? In all observed examples, the rank uniquely determines the Hilbert function for an extreme ray of $\Sigma_{2d}^\vee$.
\item If there exists an extreme ray of $\Sigma_{2d}^\vee$ with Gorenstein ideal with Hilbert function $T$, then is variety $\gor(T)$ necessarily dual to an irreducible component of $\partial_a\Sigma_{2d}$? We conjecture that this is the case.
\end{enumerate}}
%\end{oq}

\section{Hankel Matrices and Gorenstein Ideals}\label{sec:Gorenstein}

Let us fix the following notations:
We denote by $k[\ul{x}]=k[x,y,z]$ the polynomial ring over a field $k$ generated by $3$ variables. We consider it with the standard total degree grading and denote by $k[\ul{x}]_m$ the $k$-vector space of homogeneous polynomials of degree $m$, which has dimension $\binom{m+2}{2}$.

A linear functional $\ell$ on the real vector space of ternary forms of degree $2d$ is non-negative on every square if and only if the bilinear form 
\[
B_\ell\colon \left \{
\begin{array}[]{rcl}
\R[x,y,z]_d\times \R[x,y,z]_d & \to & \R \\
(f,g) & \mapsto & \ell(f\cdot g).
\end{array}
\right.
\]
is positive semi-definite. The representing matrix of this bilinear form with respect to the monomial basis is the Hankel matrix associated with $\ell$. 
Therefore, the convex cone dual to the cone $\Sigma_{2d}$ of sums of squares of polynomials is the Hankel spectrahedron:
\[
\Sigma_{2d}^\vee = \{\ell\in\R[x,y,z]_{2d}^\ast \colon (\ell(x^{\alpha+\beta}))_{\alpha,\beta} \text{ is positive semi-definite}\}.
\]

Every real point evaluation $\ev_x\colon\R[x,y,z]_{2d}\to \R$, $p\mapsto p(x)$, at $x\in\R^3$ is an extreme ray of $\Sigma_{2d}^\vee$. In fact, by the Veronese embedding of $\PP^2$ of degree $2d$, they are exactly the positive semi-definite rank $1$ Hankel matrices. We are interested in extreme rays of higher rank. These correspond to supporting hyperplanes of $\Sigma_{2d}$ which expose a face whose relative interior consists of strictly positive polynomials. Conversely, for every non-negative polynomial $p$ that is not a sum of squares, there exists an extreme ray $\R_+\ell$ of $\Sigma_{2d}^\vee$ such that $\ell(p)<0$.

\subsection{Gorenstein Ideals}
Let $\ell\in \C[\ul{x}]_{m}^\ast$ be a linear functional on ternary forms of degree $m$. To $\ell$ and every pair of positive integers $u,v\in\N$ with $u+v=m$, we associate the bilinear form
\[
B_{\ell,u,v}\colon 
\left\{
\begin{array}[h]{rcl}
\C[\ul{x}]_u\times\C[\ul{x}]_v & \to & \C \\
(p,q) & \mapsto & \ell(pq).
\end{array}\right.
\]
The representing matrices of these bilinear forms with respect to the monomial bases are called the \emph{Catalecticant matrices} of $\ell$.

\begin{Def}\label{Def:Gorenstein}
Let $\ell\in\C[\ul{x}]_{m}^\ast$ be a linear functional. We call the homogeneous ideal $I(\ell)$ of $\C[\ul{x}]$ generated by
\[
\{p\in \C[\ul{x}]_k\colon k>m \text{ or } \ell(pq)=0 \text{ for all } q\in\C[\ul{x}]_{m-k}\}
\]
the \emph{Gorenstein ideal with socle} $\ell$. We call $m$ the \emph{socle degree} of the ideal.
\end{Def}
These ideals were studied extensively in the literature, cf.~Iarrobino-Kanev \cite{IarKanMR1735271}. Our definition is probably the most direct for $0$-dimensional Gorenstein ideals, cf.~\cite[Theorem 21.6 and Exercise 21.7]{EisenbudMR1322960}.

\begin{Rem}\label{Rem:SymHilbFunc}
The degree $u$ part of the ideal is the left-kernel of the bilinear form $B_{\ell,u,v}$ for $u\leq m$. In particular, the Hilbert function of a Gorenstein ideal $I$ with even socle degree $2d$ is symmetric around $d$, i.e.~$\hilb(I,i)=\hilb(I,2d-i)$ for all $0\leq i \leq 2d$.
\end{Rem}

We can consider the set of all Gorenstein ideals with a fixed socle degree $m$ as a projective space by identifying an ideal with its socle, which is uniquely determined by the ideal up to scaling. In this projective space, we consider the set $\gor(T)$ of all Gorenstein ideals with a given Hilbert function $T$.
\begin{Prop}
The set $\gor(T)$ of all Gorenstein ideals with socle degree $m$ and Hilbert function $T$ is a quasiprojective subvariety of the projective space of all Gorenstein ideals with socle degree $m$.
\end{Prop}

\begin{proof}
The condition to have a given Hilbert function can be expressed as rank conditions on the Catalecticant matrices, namely
\[
\rk(B_{\ell,u,v}) = T(u).
\]
\end{proof}

\begin{Rem}\label{Rem:gorTQ}
(a) The quasiprojective variety $\gor(T)$ is defined over $\Q$, because the minors of the Catalecticant matrices are polynomials with coefficients in $\Z$.\\
(b) Note that a $k$-rational point $\ell\in\gor(T)$ for a subfield $k\subset \C$ is a linear functional $\ell = \ell\otimes 1 \in k[\ul{x}]_m\otimes \C$.
\end{Rem}

\begin{Def}
We call a Hilbert function $T$ \emph{permissible} if there is a Gorenstein ideal $I\subset\C[\ul{x}]$ with Hilbert function $T$.
\end{Def}

Using the Buchsbaum-Eisenbud Structure Theorem for height $3$ Gorenstein ideals (cf.~Buchsbaum-Eisenbud \cite{BE}), Diesel proved the following.
\begin{Thm}[cf.~Diesel {\cite[Theorem 1.1 and 2.7]{Die}}]\label{Thm:DieGorT}
For every permissible Hilbert function $T$, the variety $\gor(T)$ is an irreducible unirational variety.
\end{Thm}

We will use the fact that $\gor(T)$ is unirational to determine the dimension of $\gor(T)$ for special Hilbert functions $T$. In order to do this, we need the more precise information on the unirationality of $\gor(T)$ given by Diesel. The information we need is spread out over the paper Diesel \cite{Die}. We will give a short summary with references, using her notation and setup. 
\begin{Rem}\label{Rem:DieselCom}
Diesel proves that for a given permissible Hilbert function $T$ there is a minimal set (with respect to inclusion) $D_{min}=(Q,P)$ of degrees of generators $Q = \{q_1,\ldots,q_u\}$ and relations $P = \{p_1,\ldots,p_u\}$ for a Gorenstein ideal with Hilbert function $T$. We assume $q_1\leq q_2\leq\ldots\leq q_u$ and $p_1\geq p_2\geq\ldots\geq p_u$. The set $\gor_{D_{min}}$ of all Gorenstein ideals with generators of degree as specified by $Q$ is a dense subset of $\gor(T)$, see the proof of \cite[Theorem 2.7 and Theorem 3.8]{Die}. Given $D_{min}$, we consider the affine space $\A^{h(E_M)}$ of skew-symmetric matrices with entries in $\C[\ul{x}]$ where the $(i,j)$-th entry is homogeneous of degree $p_j-q_i$ ($i\neq j$) and the rational map $\pi\colon \A^{h(E_M)} \ratto \gor_{D_{min}}$ that takes a matrix to the Gorenstein ideal generated by its Pfaffians. This statement uses the Buchsbaum-Eisenbud Structure Theorem, cf.~\cite{Die}, p.~367 and p.~369.
Given a Hilbert function $T$, the set $D_{min}$ of degrees of generators and relations for $T$ is determined in a combinatorial way: Given the socle degree $m$ and the minimal degree $k$ of a generator of the ideal, there is a one-to-one correspondence between permissible Hilbert functions of order $k$ and self-complementary partitions of $2k$ by $m-2k+2$ blocks, cf.~\cite[Proposition 3.9]{Die}. These partitions give the maximum number of generators, which is $2k+1$, cf.~\cite[Theorem 3.3]{Die}. To refine these sequences to $D_{min}$, we iteratively delete pairs $(q_i,q_j)$ from $Q$ and $(p_i,p_j)$ from $P$ whenever they satisfy $r_i+r_j=p_i+p_j-q_i-q_j=0$, cf.~\cite{Die}, p.~380.
\end{Rem}

We are particularly interested in Gorenstein ideals with socle in even degree $2d$ with the property that the middle Catalecticant has corank $4$, i.e.~rank $\binom{d+2}{2}-4$. The proof of the following statement is analogous to the proof of Diesel \cite[Theorem 4.4]{Die}.
\begin{Lem}\label{Lem:Corank4dim}
Let $d\geq 4$ be an integer. The projective variety $X_{-4}$ of middle Catalecticant matrices of corank at least $4$, i.e.~of rank at most $\binom{d+2}{2}-4$, is irreducible of codimension $10$ in the space of middle Catalecticant matrices. It has degree $\prod_{\alpha = 0}^3 \binom{N+\alpha}{4-\alpha}/\binom{2\alpha +1}{\alpha}$, where $N = \binom{d+2}{2}$. In particular, it is defined by the $(\binom{d+2}{2}-3)$-minors of the generic middle Catalecticant matrix.
\end{Lem}

\begin{proof}
Let $N = \binom{d+2}{2}$. The quasiprojective variety $S_{-4}$ of symmetric $N\times N$-matrices of rank $N-4$ has codimension $10$ in the projective space of the vector space of symmetric $N\times N$-matrices. 
Therefore the intersection $X_{-4}$ of $S_{-4}$ with the subspace of middle Catalecticant matrices has codimension at most $10$ in this linear space. We will show, that it has codimension exactly $10$ by counting dimensions of the possible $\gor(T)$ using their unirationality. We will use the setup and notation used by Diesel \cite{Die}, see also \ref{Rem:DieselCom}.

There are only two possible Hilbert functions for a Gorenstein ideal $I$ with socle degree $2d$ and $\hilb(I,d)=\binom{d+2}{2}-4$ by their symmetry, namely
\[
T_1=(1,3,6,\ldots,\binom{d+1}{2},\binom{d+2}{2}-4,\ldots),
\]
which corresponds to the case of four generators in degree $d$ and no generators of lower degree, and
\[
T_2 = (1,3,6,\ldots,\binom{d+1}{2}-1,\binom{d+2}{2}-4,\ldots),
\]
which corresponds to the case of one generator of degree $d-1$ and one generator of degree $d$. More precisely, these two Hilbert functions correspond to the self-complementary partitions of $2\times 2d$ resp.~$4\times(2d-2)$ blocks shown in Figure \ref{fig:GorensteinPartitions} by the correspondence explained in Diesel \cite[section 3.4, in paricular Proposition 3.9]{Die}.

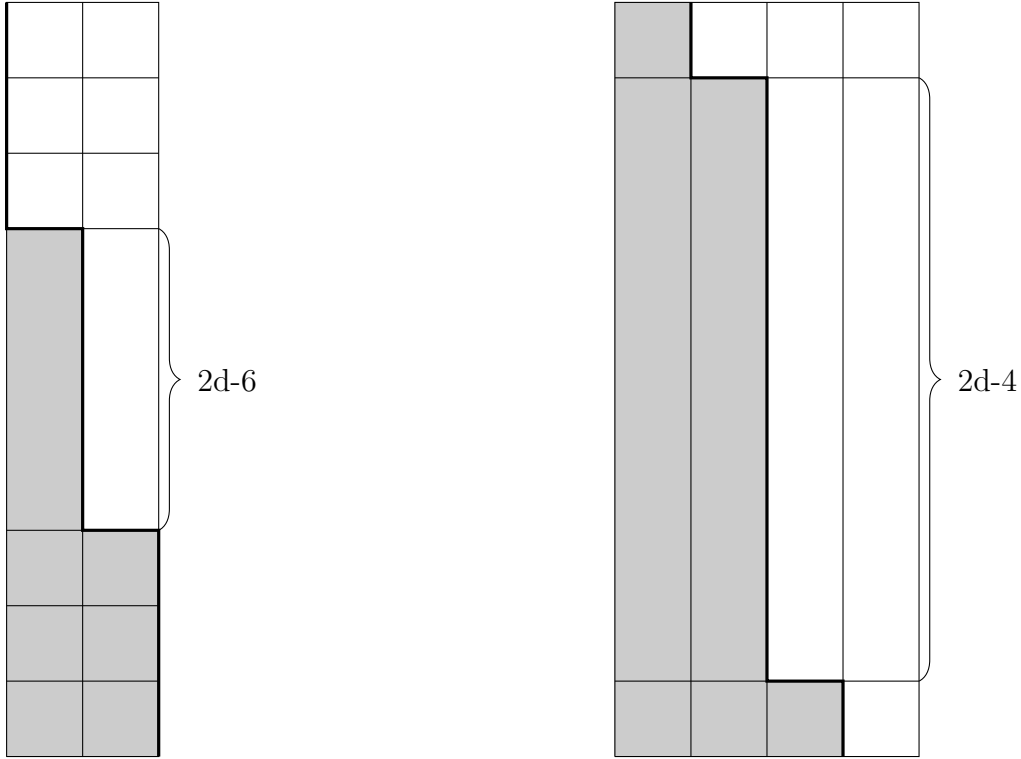
\begin{figure}[h]
\begin{center}
\begin{tikzpicture}
\filldraw[color = black!20!white] (-6,0) rectangle (-5,7);
\filldraw[color = black!20!white] (-5,0) rectangle (-4,3);
\draw[black] (-6,0) -- (-6,10) -- (-4,10) -- (-4,0) -- (-6,0);
\draw (-5,0) -- (-5,10);
\draw (-6,1) -- (-4,1);
\draw (-6,2) -- (-4,2);
\draw (-6,3) -- (-4,3);
\draw (-6,7) -- (-4,7);
\draw (-6,8) -- (-4,8);
\draw (-6,9) -- (-4,9);
\draw[very thick] (-6,10) -- (-6,7) -- (-5,7) -- (-5,3) -- (-4,3) -- (-4,0);
\draw[decorate,decoration={brace,amplitude=8pt}] (-4,7) -- (-4,3) node[midway,xshift=0.9cm] {2d-6};

\filldraw[color = black!20!white] (2,0) rectangle (3,10);
\filldraw[color = black!20!white] (3,0) rectangle (4,9);
\filldraw[color = black!20!white] (4,0) rectangle (5,1);
\draw[black] (2,0) -- (2,10) -- (6,10) -- (6,0) -- (2,0);
\draw (3,0) -- (3,10);
\draw (4,0) -- (4,10);
\draw (5,0) -- (5,10);
\draw (2,9) -- (6,9);
\draw (2,1) -- (6,1);
\draw[very thick] (3,10) -- (3,9) -- (4,9) -- (4,1) -- (5,1) -- (5,0);
\draw[decorate,decoration={brace,amplitude=8pt}] (6,9) -- (6,1) node[midway,xshift=0.9cm] {2d-4};
\end{tikzpicture}
\end{center}
\caption{The partition on the right of $2d\times 2$ blocks corresponds to $T_1$, the partition on the left of $(2d-2)\times 4$ blocks to $T_2$.}
\label{fig:GorensteinPartitions}
\end{figure}

We first consider $T_1$. The sequence of degrees of the generators for the minimal set $D_{min}$ is in this case different for $d=4$ and $d\geq 5$, namely $(4,4,4,4,6)$ for $d=4$ and $(d,d,d,d,d+1,\ldots,d+1)$ with $(2d-9)$ many generators of degree $d+1$ for $d\geq 5$, cf.~Remark \ref{Rem:DieselCom}. Since $q_i+p_i = 2d+3$, the degree matrices are
\[
\left(
\begin{array}[l]{ccccc}
0 & 3 & 3 & 3 & 1 \\
  & 0 & 3 & 3 & 1 \\
  &   & 0 & 3 & 1 \\
  &   &   & 0 & 1 \\
  &   &   &   & 0 \\
\end{array}\right),
\left(
\begin{array}[r]{cccccccc}
0 & 3 & 3 & 3 & 2 & \cdots & \cdots & 2 \\
  & 0 & 3 & 3 & \vdots &   &        & \vdots \\
  &   & 0 & 3 & \vdots &   &        & \vdots \\
  &   &   & 0 & 2 & \cdots & \cdots & 2 \\
  &   &   &   & 0 & 1      & \cdots & 1 \\
  &   &   &   &   & 0      & \ddots & \vdots \\
  &   &   &   &   &        & \ddots & 1 \\
  &   &   &   &   &        &        & 0
\end{array}\right)
\]
where the right one is of size $(2d-5)\times (2d-5)$. Every entry of the matrix can be generically chosen among the forms of the indicated degree and its Pfaffians will generate a Gorenstein ideal with Hilbert function $T_1$. Therefore, for $d=4$, we have $h(E_M)=6\dim(\C[\ul{x}]_3) + 4\dim(\C[\ul{x}]_1)=72$ and for $d\geq 5$ we have
\begin{eqnarray*}
h(E_M)& = & 6\dim(\C[\ul{x}]_3)+4(2d-9)\dim(\C[\ul{x}]_2) \\
& & + \binom{2d-9}{2}\dim(\C[\ul{x}]_1)\\
& = & 6d^2-9d-21.
\end{eqnarray*}
This is an overcount of the dimension of $\gor_{D_{min}}$ because for every choice of generators of a given ideal we get a matrix with these generators as Pfaffians. So for $d=4$, we choose a basis of a $4$-dimensional subspace of forms of degree $4$ and one generator of degree $6$ from a $\dim(\C[\ul{x}]_6) - T_1(2)=22$-dimensional space. Therefore we overcount the dimension of $\gor_{D_{min}}$ by at least $4^2+22=38$ and the dimension of $\gor(T)$ is at most $34$. Since $\dim(\P(\C[\ul{x}]_8^\ast)) = 44$, its codimension is at least $10$. For $d\geq 5$, we choose a basis of a $4$-dimensional subspace of forms of degree $d$ and $2d-9$ linearly independent generators from a space of dimension $\dim(\C[\ul{x}]_{d+1})-T_1(d-1)=2d+3$. The overcount in this case is at least $4^2+(2d-9)(2d+3)$ and the dimension of $\gor_{D_{min}}$ is at most $2d^2+3d-10$. The projective dimension of the space of middle Catalecticant matrices is $\dim(\C[\ul{x}]_{2d}^\ast)-1=2d^2+3d$, which again implies that the codimension of $\gor(T_1)$ is at least $10$. From the fact that it can be at most $10$, it follows that it is exactly $10$.

We now repeat the count for the Hilbert function $T_2$. In this case, $D_{min} = \{Q_{min},P_{min}\} = \{(d-1,d,d+1,d+1,\ldots,d+1),(d+4,d+3,d+2,d+2,\ldots,d+2)\}$ with $(2d-5)$ times the entry $d+1$ in $Q_{min}$ and $d+2$ in $P_{min}$, cf.~Figure \ref{fig:GorensteinPartitions}. Therefore, the degree matrix is
\[
\left(
\begin{array}[h]{cccccc}
0 & 4 & 3 & \cdots & \cdots & 3 \\
  & 0 & 2 & \cdots & \cdots & 2 \\
  &   & 0 & 1      & \cdots & 1 \\
  &   &   & 0      & \ddots & \vdots \\
  &   &   &        &        & 1 \\
  &   &   &        &        & 0
\end{array}\right)
\]
which is of size $(2d-3)\times (2d-3)$. We compute
\begin{eqnarray*}
h(E_M) & = & \dim(\C[\ul{x}]_4)+(2d-5)\dim(\C[\ul{x}]_3)+(2d-5)\dim(\C[\ul{x}]_2) \\
& & +\binom{2d-5}{2}\dim(\C[\ul{x}]_1) \\
& = & 6d^2-d-20.
\end{eqnarray*}
 Here we choose one generator of degree $d-1$, one generator of degree $d$ from a $4$-dimensional space and $(2d-5)$ generators from a $\dim(\C[\ul{x}]_{d+1})-T_2(d-1)=(2d+4)$-dimensional space. Therefore the dimension of $\gor(T_2)$ is at most $6d^2-d-20-4-(2d-5)(2d+4)=2d^2+d-4$. The codimension is at least $2d+4\geq 12$.
So $\gor(T_2)$ cannot be an irreducible component of $X_{-4}$ and we conclude that $\gor(T_2)\subset \cl(\gor(T_1))$.

In summary, $\gor(T_1)$ is a dense subset of $X_{-4}$ and $X_{-4}$ is irreducible (cf.~Diesel \cite[Theorem 2.7]{Die}) and has the expected codimension $10$ in the space of middle Catalecticant matrices. Therefore, the intersection $X_{-4}$ of the variety $S_{-4}$ of symmetric $N\times N$ matrices of corank at least $4$ and the linear space of Hankel matrices is generically transversal and hence preserves the degree, i.e.~$\deg(X_{-4}) = \deg(S_{-4})$. The degree of $S_{-4}$ was computed in Harris-Tu \cite[Proposition 12(b)]{HarrisTu} and is equal to
\[
\prod_{\alpha=0}^3 \binom{N+\alpha}{4-\alpha}/\binom{2\alpha+1}{\alpha}.
\]
\end{proof}

The tangent space to the quasiprojective variety $\gor(T)$ for a permissible Hilbert function $T$ at a Gorenstein ideal $I$ can be described in terms of the ideal. We identify $\C[\ul{x}]_m$ with its dual space by using the apolar bilinear form, i.e.~we identify a monomial $x^\alpha\in\C[\ul{x}]_m$ with the linear form $p\mapsto \frac{1}{\alpha!} \frac {\partial^{|\alpha|}}{\partial x^\alpha}p$ that takes a polynomial $p=\sum p_\beta x^\beta$ to $p_\alpha$. Using this identification, we can state a characterisation of the tangent space to $\gor(T)$ at an ideal $I$ in terms of this ideal.
\begin{Thm}[Iarrobino-Kanev {\cite[Theorem 3.9 and 4.21]{IarKanMR1735271}}]
\label{Thm:GorTTangentSpace}
Let $T$ be a permissible Hilbert function. The quasiprojective variety $\gor(T)$ is smooth.
Let $\ell\in\C[\ul{x}]_m^\ast$ be a linear functional such that the corresponding Gorenstein ideal $I=I(\ell)$ has Hilbert function $T$. Then the tangent space to $\gor(T)$ at $\ell$ is
\[
( (I^2)_m)^\perp\subset \C[\ul{x}]_m.
\]
\end{Thm}

The irreducible variety $X_{-4}$ of middle Catalecticant matrices of corank at least $4$ is defined over $\Q$ and we will later show that it has a smooth rational point, i.e.~a point with rational coordinates. Therefore, the real points of $X_{-4}$ are Zariski-dense in it and the above statement of Theorem \ref{Thm:GorTTangentSpace} also applies to real points of $X_{-4}$, cf.~\cite[Section 2.8]{BochnakMR1659509}.

\section{Extreme Rays of Maximal Rank and Positive Gorenstein Ideals}\label{sec:ExtRays}
In this section, we recapitulate bounds on the rank of Hankel matrices of extreme rays of $\Sigma_{2d}^\vee$ which are not point evaluations. The lower bound and its tightness are proved in Blekherman \cite[Theorem 2.1]{BlePositiveGorensteinIdeals}. We constructively establish tightness of the upper bound. We show that the Zariski closure of the set of extreme rays is the variety of Hankel matrices of corank at least $4$, which is irreducible; in particular, it is (at least set-theoretically) defined by the symmetric $r\times r$ minors of the generic Hankel matrix, where $r=\binom{d+2}{2}-3$.

To a linear functional $\ell \in \R[\ul{x}]_m^\ast$, we associate the bilinear form 
\[
B_{\ell}\colon 
\left\{
\begin{array}[h]{rcl}
\R[\ul{x}]_d\times\R[\ul{x}]_d & \to & \R \\
(p,q) & \mapsto & \ell(pq),
\end{array}\right.
\]
whose representing matrix with respect to the monomial bases is called the \emph{Hankel matrix} of $\ell$.

One of the main results of Blekherman is a characterisation of extreme rays of $\Sigma_{2d}^\vee$ by the associated Gorenstein ideals.
\begin{Prop}[Blekherman {\cite[Lemma 2.2]{BleNP}} and {\cite[Proposition 4.2]{BlePositiveGorensteinIdeals}}]
\label{Prop:ExtremeRays}
(a) A linear functional $\ell\in\R[\ul{x}]_{2d}^\ast$ spans an extreme ray of $\Sigma_{2d}^\vee$ if and only if the bilinear form $B_\ell$ is positive semi-definite and the degree $d$ part $I(\ell)_d$ of the Gorenstein ideal $I(\ell)$ is maximal with respect to inclusion over all Gorenstein ideals with socle degree $2d$.\\
(b) Let $I$ be a Gorenstein ideal with socle degree $2d$. Then $I_d$ is maximal with respect to inclusion over all Gorenstein ideals with socle degree $2d$ if and only if the degree $2d$ part of the ideal generated by $I_d$ is a hyperplane in $\R[\ul{x}]_{2d}$. In this case, it is equal to $I_{2d}$.
\end{Prop}

Lower bounds on the ranks for extreme rays were established by Blekherman.
\begin{Thm}[Blekherman {\cite[Theorem 2.1]{BlePositiveGorensteinIdeals}}]\label{Thm:ExtremeRays}
Let $d\geq 3$ and $\ell\in \Sigma_{2d}^\vee$ and suppose $\R_+\ell$ is an extreme ray. Then the rank $r$ of $B_\ell$ is $1$, in which case $\ell$ is a point evaluation, or its rank is at least $3d-2$.
These bounds are tight and extreme rays $\Sigma_{2d}^\vee$ of rank $3d-2$ can be explicitly constructed.
\end{Thm}

From Blekherman's work, we can easily deduce an upper bound.
\begin{Thm}\label{Thm:MaxRank}
Let $\ell\in\Sigma_{2d}^\vee$, $d\geq 4$ and suppose $\R_+\ell$ is an extreme ray. The rank of $B_\ell$ is at most $\binom{d+2}{2}-4$, i.e.~the corank is at least $4$.
\end{Thm}

\begin{proof}
Since $\R_+\ell$ is an extreme ray, we know that the degree $2d$ part of the ideal generated by $I(\ell)_d$ is a hyperplane in the space of forms of degree $2d$.
The dimension of the space $\R[\ul{x}]_d I(\ell)_d$ is bounded by $\dim(\R[\ul{x}]_d) \dim(I(\ell)_d)=\binom{d+2}{2}\crk(B_\ell)$. In case $\crk(B_\ell)\leq 3$ and $d\geq 5$, this bound is smaller than the dimension $\binom{2d+2}{2}-1$ of a hyperplane in $\R[\ul{x}]_{2d}$. The case $\crk(B_\ell) \leq 3$ and $d=4$ needs a more precise count: Suppose that $\crk(B_\ell)=3$ and the kernel of $B_\ell$ is generated by $f_1,f_2,f_3$. Then the dimension of the space $\R[\ul{x}]_4 I(\ell)_4$ is bounded by $3\dim(\R[\ul{x}]_4)-3 = 42<45-1 = \dim(\R[\ul{x}]_8) -1$ because there are the $3$ obvious relations, namely $f_if_j-f_jf_i = 0$ for $i\neq j$. 
\end{proof}

\begin{Rem}
The upper bound in the case $d=3$ is corank $3$, which agrees with the lower bound.
\end{Rem}

A main tool in this section is the Cayley-Bacharach Theorem.
\begin{Thm}[Cayley-Bacharach, cf.~Eisenbud-Green-Harris {\cite[CB5]{EisGreHar}}]
Let $X_1,X_2\subset\P^2$ be plane curves defined over $\R$ of degree $d$ and $e$ intersecting in $d\cdot e$ points. Set $s= d+e-3$ and decompose $X_1\cap X_2 = \Gamma_1\cup \Gamma_2$ into two disjoint sets defined over $\R$. Then for all $k\leq s$, the following equality holds
\begin{eqnarray*}
& \dim(\I(\Gamma_1)_k) - \dim(\I(X_1\cap X_2)_k) = \\
& |\Gamma_2| - \dim \lspan \{ {\rm Re}\ev_x, {\rm Im}\ev_x \in\R[\ul{x}]_{s-k}^\ast\colon x\in\Gamma_2\}.
\end{eqnarray*}
The left hand side is the dimension of the space of forms of degree $k$ vanishing on $\Gamma_1$ modulo the subspace of forms vanishing in every point of $X_1\cap X_2$. The right hand side is the linear defect of point evaluations on forms of dual degree $s-k$ at points of $\Gamma_2$.
\end{Thm}

Probably the most famous instance of this theorem is the following application to the complete intersection of two cubic curves, stated here for a totally real intersection.
\begin{Exm}
Suppose $X_1,X_2\subset\P^2$ are plane cubic curves intersecting in $9$ points. Then $d=e=3$ and so $s=3$. Pick $\Gamma_2 = \{P\}$ for any intersection point $P$ and put $\Gamma_1 = (X_1\cap X_2) \setminus \{P\}$. Let us consider $k=3$ and compute the right hand side of the Cayley-Bacharach equality: Since $\dim \lspan\{\ev_x\in\R[\ul{x}]_0^\ast\colon x\in\Gamma_2\} = 1$, we conclude
\[
\dim(\I(\Gamma_1)_3) - \dim(\I(X_1\cap X_2)_3) = 0,
\]
which means that every cubic form that vanishes in the $8$ points of $\Gamma_1$ also vanishes at the ninth point $P$ of the intersection. In other words, the point evaluation $\ev_P\in\R[\ul{x}]_3^\ast$ lies in the subspace $U_{\Gamma_1}$ spanned by the eight point evaluations $\{\ev_x\in\R[\ul{x}]_3^\ast\colon x\in\Gamma_1\}$. The annihilator of $U_{\Gamma_1}$ is the $2$-dimensional subspace of $\R[\ul{x}]_3$ spanned by the defining equations of $X_1$ and $X_2$. Since this is true for any point $P\in X_1\cap X_2$, we conclude, that there is a unique linear relation among the point evaluations $\{\ev_x\in\R[\ul{x}]_3^\ast\colon x\in X_1\cap X_2\}$ and all coefficients of this relation are non-zero.
\end{Exm}

Using the Cayley-Bacharach Theorem, we will first show that there are extreme rays of corank $4$ under the following constraint on the degree. We will get rid of this constraint in Lemma \ref{Lem:ExRayC4Per}.
\begin{Const}\label{Const:ellipse}
Let $d\geq 4$. There is a unique conic $C$ going through the following six points in the plane: $(0,0),(1,0),(0,1),(d-1,d-1),(d-2,d-1),(d-1,d-2)$; its equation is given by $C = \V(x^2+y^2-\frac{2(d-2)}{d-1}xy-x-y)$. From now on, we assume that this conic does not go through any other integer point. The only exceptional cases in the interval $\{4,5,\ldots,100\}$ are: $9,19,21,29,33,34,36,40,49,51,57,61,73,78,79,81,89,99$.
\end{Const}

\begin{Prop}\label{Prop:CB}
Set $L_1=\prod_{j=0}^{d-1}(x-jz)$ and $L_2=\prod_{j=0}^{d-1}(y-jz)$ and let $\Gamma=\V(L_1,L_2)=\{(j:k:1)\colon j,k=0,\ldots,d-1\}$ be the intersection of their zero sets in $\P^2$. Split these points into
\[
\Gamma_2 = \{(x:y:1)\colon x+y=2\}\cup\bigcup_{j=1}^{d-4}\{(x:y:1)\colon x+y=d+j\}
\]
and $\Gamma_1=\Gamma\setminus\Gamma_2$.
Then there is a unique linear relation $\sum_{v\in\Gamma_1} u_v{\rm ev}_v=0$ among the point evaluations on forms of degree $d$ at points of $\Gamma_1$ and all coefficients $u_v\in\R$ in this relation are non-zero.
The set of all forms of degree $d$ vanishing on $\Gamma_1$ is a $3$-dimenional space spanned by $L_1,L_2$ and a form $p$ which is non-zero at any point of $\Gamma_2$.
\end{Prop}
See Figure \ref{fig:ExtRayC4} for the case $d=5$ and Figure \ref{fig:perturbation} for the case $d=9$.

\begin{proof}
First observe that there is a unique (up to scaling) form of degree $d-3$ vanishing on $\Gamma_2$, namely $(x+y-2z)\prod_{j=1}^{d-4}(x+y-(d+j)z)$, the product of diagonals defining $\Gamma_2$: Indeed, suppose $f$ is a form of degree $d-3$ vanishing on $\Gamma_2$, then it intersects the line $x+y=d+1$ in $d-2$ integer points. Therefore it vanishes identically on it and we can divide $f$ by this linear polynomial and get a form of degree $d-4$ vanishing on $d-3$ points on the line $x+y=d+2$. Inductively, we conclude that $f$ is (again up to scaling) the claimed product of linear forms. Therefore, by the Cayley-Bacharach Theorem, the space of forms of degree $d$ vanishing on $\Gamma_1$ is $3$-dimensional,
so it is spanned by $L_1, L_2$ and a third form $p$. We will explicitly construct this form: Let $p$ be the product of the linear forms $x+y-jz$ for $j=3,\ldots,d$ and of the ellipse $\V(x^2+y^2-\frac{2(d-2)}{(d-1)}xy-x-y)$ passing through the six points $(0,0),(1,0),(0,1),(d-2,d-1),(d-1,d-1)$ and $(d-1,d-2)$. By construction, $p$ vanishes on $\Gamma_1$, is of degree $d$ and does not vanish on all of $\Gamma$. Therefore $\{L_1,L_2,p\}$ is a basis of the space of forms of degree $d$ vanishing on $\Gamma_1$. By assumption on $d$, the form $p$ does not vanish on any point of $\Gamma$ other than the six mentioned above. 

Note that $|\Gamma_1| = \binom{d+2}{2}-2$, because $|\Gamma_1| = d^2 - |\Gamma_2| = d^2 - (3+\sum_{j=1}^{d-4}(d-1-j))=d^2-\binom{d-1}{2}$.
So the fact that the space of forms of degree $d$ vanishing on $\Gamma_1$ is $3$-dimensional implies that there is a unique linear relation among the point evalutaions on forms of degree $d$ at points of $\Gamma_1$. To see that all coefficients $u_v$ in the relation $\sum_{v\in\Gamma_1}u_v{\rm ev}_v=0$ are non-zero, note that the unique form $f$ of degree $d-3$ vanishing on $\Gamma_2$ does not vanish on any point of $\Gamma_1$. Therefore, there is no form of degree $d-3$ vanishing on $\Gamma_2\cup\{v_0\}$ for any $v_0\in\Gamma_1$ and Cayley-Bachrach implies that the point evaluations $\{ {\rm ev}_v\colon v\in \Gamma_1\}\setminus\{ {\rm ev}_{v_0}\}$ are linearly independent.
\end{proof}

\begin{Lem}\label{Lem:ExtremeRaysC4}
There is an extreme ray $\R_+\ell$ of $\Sigma_{2d}^\vee$ such that $B_\ell$ has corank $4$. The Hilbert function of the ideal $I(\ell)$ is $\hilb(I(\ell),j)=\binom{j+2}{2} = \hilb(I(\ell),2d-j)$ for $0\leq j<d$ and $\hilb(I(\ell),d)=\binom{d+2}{2}-4$.
\end{Lem}

\begin{proof}
Let $L_1,L_2,p$ be as in Proposition \ref{Prop:CB} and consider the splitting $\V(L_1,L_2) = \Gamma =  \Gamma_1 \cup \Gamma_2$ of the points defined there. Pick a point $P\in\Gamma_1$ and set $\Lambda = \Gamma_1\setminus\{P\}$. We claim that the linear functional
\[
\ell = \sum_{v\in \Lambda} \ev_v - \frac{u_P^2}{\sum_{v\in\Lambda}u_v^2} \ev_P,
\]
where $u_v$ are the coefficients of the Cayley-Bacharach relation as in Proposition \ref{Prop:CB}, is an extreme ray of $\Sigma_{2d}^\vee$ and that the corresponding Hankel matrix $B_\ell$ has corank $4$.

First note that $B_\ell$ is positive semi-definite because
\begin{eqnarray*}
\ell(f^2) & = & \sum_{v\in\Lambda} f(v)^2 - \frac{u_P^2}{\sum_{v\in\Lambda} u_v^2} f(P)^2 \\
& = & \sum_{v\in\Lambda} f(v)^2 - \frac{u_P^2}{\sum_{v\in\Lambda} u_v^2} \frac{1}{u_P^2} \left(\sum_{v\in\Lambda} u_vf(v)\right)^2 \\
& = & \|(f(v)_{v\in\Lambda}\|^2 - \left| \left\langle \frac{1}{\|(u_v)_{v\in\Lambda}\|} (u_v)_{v\in\Lambda} , (f(v))_{v\in\Lambda} \right\rangle \right|^2 \\
& \geq & 0
\end{eqnarray*}
by the Cauchy-Schwarz inequality for all polynomials $f\in\R[\ul{x}]_d$. More precisely, $\ell(f^2)$ is zero for a form $f$ not identically zero on $\Gamma$ if and only if $f(v)=\alpha u_v$ for all $v\in\Lambda$ and some $\alpha\in\R^\ast$. Therefore, the degeneration space of the Hankel matrix is spanned by $L_1,L_2,p$ and the form uniquely determined (modulo $L_1,L_2,p$) by $f(v)=u_v$ for all $v\in\Lambda$; it has dimension $4$ as desired. Indeed, the form $f$ is uniquely determined because $\{\ev_x\in (\R[\ul{x}]_d/\lspan(L_1,L_2,p))^\ast\colon x\in\Lambda\}$ is a basis.

We now prove extremality of $\ell$ in $\Sigma_{2d}^\vee$ by checking the characterisation that $I(\ell)_d$ generates a hyperplane in the vector space of forms of degree $2d$, cf.~Blekherman \cite[Proposition 4.2]{BlePositiveGorensteinIdeals}. As a first step, we show that $\langle L_1,L_2,p \rangle_{2d-3}$ has codimension $|\V(L_1,L_2,p)|=|\Gamma_1|$. So suppose $a_1 L_1+a_2 L_2 + bp=0$, where $a_1,a_2,b\in\R[\ul{x}]_{d-3}$ are forms of degree $d-3$. By evaluating at points of $\Gamma_2$, we conclude that $b$ is the uniquely determined form of degree $d-3$ vanishing on $\Gamma_2$, cf.~proof of Proposition \ref{Prop:CB}. Since $L_1$ and $L_2$ are coprime, this is a unique syzygy and we conclude
\[
\dim(\langle L_1,L_2,p\rangle_{2d-3}) = 3\binom{d-1}{2} -1 = \frac32(d^2-3d+2)-1,
\]
which means codimension $|\Gamma_1|=d^2-\binom{d-1}{2}$ in $\R[\ul{x}]_{2d-3}$.
In particular, the codimension of $\langle L_1,L_2,p\rangle_{2d}$ in $\R[\ul{x}]_{2d}$ is also $|\Gamma_1|$ because the point evaluations $\{\ev_v\colon v\in\Gamma_1\}$ are linearly independent on forms of degree $2d-3$ and consequently also on forms of degree $2d$.

Now suppose $a_1L_1 + a_2L_2+bp + cf=0$ for forms $a_1,a_2,b,c\in\R[\ul{x}]_d$ of degree $d$. Evaluation at points of $\Gamma_1$ implies that $c$ lies in the span of $L_1,L_2,p$. So we have three syzygies and the codimension of $\langle L_1,L_2,p,f\rangle_{2d}$ is $|\Gamma_1| - \binom{d+2}{2}+3=1$, as desired.
\end{proof}

\begin{Exm}\label{Exm:Hankel}
We follow the construction in Proposition \ref{Prop:CB} and Lemma \ref{Lem:ExtremeRaysC4} in the case $d=5$. Then $\Gamma=\V(L_1,L_2)$ consists of the $25$ points $(i:j:1)\in\P^2$ where $i,j=0,\ldots,4$, see Figure \ref{fig:ExtRayC4}. The six points on the two lines $x+y=2$ and $x+y=6$ are the points of $\Gamma_2$. Indeed, the point evaluations at the $19$ points of $\Gamma_1 = \Gamma\setminus \Gamma_2$ on forms of degree $5$ satisfy a unique linear relation, namely

$\begin{pmatrix}
-1 &   3 &   0 &  -5 &  3 \\
 3 & -16 &  18 &   0 & -5 \\
 0 &  18 & -36 &  18 &  0 \\
-5 &   0 &  18 & -16 &  3 \\
 3 &  -5 &   0 &   3 & -1 \\
\end{pmatrix}$,

where the $(i,j)$-th entry of this matrix is the coefficient of the point evaluation at $(5-i:j-1:1)$ in the linear relation, i.e.~visually, it is the coefficient corresponding to the points in the $5\times 5$-grid seen in Figure \ref{fig:ExtRayC4}.
\begin{figure}[h]
\begin{center}
\begin{tikzpicture}
\filldraw (0,0) circle(2pt);
\filldraw (0,1) circle(2pt);
\filldraw (0,2) circle(2pt);
\filldraw (0,3) circle(2pt);
\filldraw (0,4) circle(2pt);
\filldraw (1,0) circle(2pt);
\filldraw (1,1) circle(2pt);
\filldraw (1,2) circle(2pt);
\filldraw (1,3) circle(2pt);
\filldraw (1,4) circle(2pt);
\filldraw (2,0) circle(2pt);
\filldraw (2,1) circle(2pt);
\filldraw (2,2) circle(2pt);
\filldraw (2,3) circle(2pt);
\filldraw (2,4) circle(2pt);
\filldraw (3,0) circle(2pt);
\filldraw (3,1) circle(2pt);
\filldraw (3,2) circle(2pt);
\filldraw (3,3) circle(2pt);
\filldraw (3,4) circle(2pt);
\filldraw (4,0) circle(2pt);
\filldraw (4,1) circle(2pt);
\filldraw (4,2) circle(2pt);
\filldraw (4,3) circle(2pt);
\filldraw (4,4) circle(2pt);
\draw (-1,3) -- (3,-1);
\draw (1,5) -- (5,1);
\end{tikzpicture}
\end{center}
\caption{The construction of an extreme ray of $\Sigma_{2d}^\vee$ of corank $4$ for $d=5$.}
\label{fig:ExtRayC4}
\end{figure}
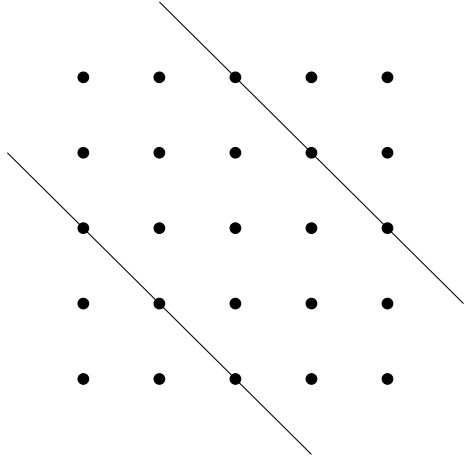
The $21\times 21$ Hankel matrix can be exactly computed using a computer algebra system. In Mathematica, the following code will do the job:

\begin{verbatim}
d = 5;
m = MonomialList[(x+y+z)^d];
q1 = x (x - z) (x - 2 z) (x - 3 z) (x - 4 z);
q2 = y (y - z) (y - 2 z) (y - 3 z) (y - 4 z);
Pevalall = Solve[{q1 ==  0, q2 ==  0, z ==  1}, {x, y, z}];
Pevalfoo = Select[Pevalall, ({y + x - 2 z} /. #) != {0} &];
Peval = Select[Pevalfoo, ({y + x - 6 z} /. #) != {0} &];
Peval0 = Drop[Peval, -1];

evals = m/.Peval;
CBrel = NullSpace[Transpose[evals]];
CB = Transpose[{Drop[CBrel[[1]],-1]}];

H = Transpose[{m}].{m};
P = Peval0;
Q = Sum[H /. P[[i]], {i, 1, Length[P]}];
Qp = H /. Peval[[-1]];
l = Norm[CB]^2;
Hankel = Q - 1/l (rel[[1]][[-1]])^2 Qp;
\end{verbatim}
We set up the monomial basis \texttt{m} and the totally real complete intersection of $25$ points, where $\texttt{q1}=L_1$ and $\texttt{q2} = L_2$. The two lines using the \texttt{Select}-command remove the points on the two diagonals $x+y=2$ and $x+y=6$, so $\Gamma_1 = \texttt{Peval}$. With the \texttt{Drop}-command, we remove one of the points from the list. 
The next three lines compute the unique Cayley-Bacharach relation \texttt{CBrel} on the point evaluations at \texttt{Peval}. We need \texttt{CB} when we solve the linear relation for the point evaluation at \texttt{Peval[[-1]]}.
The matrix \texttt{H} is the general Hankel matrix and $Q$ is the Hankel matrix of the linear functional $\sum_{v\in\Gamma_1\setminus \texttt{Peval[[-1]]}} \ev_v$ and \texttt{Qp} the Hankel matrix of the point evaluation at \texttt{Peval[[-1]]}. So \texttt{Hankel} is the Hankel matrix of the extreme ray that we constructed.
\end{Exm}

\begin{Rem}
Note that the proof of the Lemma \ref{Lem:ExtremeRaysC4} shows that the face of the cone $\Sigma_{2d}$ of sums of squares exposed by the constructed extreme ray consists of the sums of squares of polynomials in $I(\ell)_d$.
\end{Rem}

The fact that the conic vanishes in additional integer points on the $d\times d$ grid defined by the products of linear forms $L_1$ and $L_2$ in Proposition \ref{Prop:CB} destroys the extremality of the constructed linear functional because we get additional syzygies among the generators of the corresponding Gorenstein ideal. In order to deal with this problem, we will make a perturbation to our point arrangement. First, we want to observe the following fact, which motivates why we should be able to get around this obsatcle by perturbation:
\begin{Rem}\label{Rem:CBApplication}
Consider the setup in Proposition \ref{Prop:CB} and suppose the conic $C$ vanishes in additional integer points in the $d\times d$ integer grid $\Gamma=\V(L_1)\cap \V(L_2)$. Pick such a point $P\in\Gamma$. Then every form of degree $d$ vanishing on $\Gamma_1$ will also vanish at $P$ because $L_1$, $L_2$ and the third form $p$, which is the product of lines and the conic, form a basis of this space. By the Theorem of Cayley-Bacharach applied to $\Gamma = \Gamma_1'\cup\Gamma_2'$ for $\Gamma_1' = \Gamma_1 \cup\{P\}$ and $\Gamma_2' = \Gamma_2\setminus\{P\}$, there is a unique linear relation among the point evaluations at points of $\Gamma_2'$ on forms of degree $d-3$. In particular, the coefficient of the point evaluation at $P$ in the unique linear relation among point evaluations at $\Gamma_2$ on forms of degree $d-3$ is zero. The converse is also true by Cayley-Bacharach, so we have:\\
The conic $C$ vanishes in a point in $P\in\Gamma_2$ if and only if the coefficient of the point evaluation at $P$ in the unique linear relation among $\{\ev_v\in \R[\ul{x}]_{d-3}^\ast \colon v\in\Gamma_2\}$ is zero. This seems to be a non-generic property and we will indeed show that we can make all coefficients in the linear relation among these point evaluations non-zero by a careful perturbation of $L_1$ and $L_2$.
\end{Rem}

We now drop the assumptions on $d$ made in \ref{Const:ellipse} and prove Lemma \ref{Lem:ExtremeRaysC4} for all $d\geq 4$:
\begin{Lem}\label{Lem:ExRayC4Per}
For any $d\geq 4$, there is an extreme ray $\R_+\ell$ of $\Sigma_{2d}^\vee$ such that $B_\ell$ has corank $4$. The Hilbert function of the ideal $I(\ell)$ is $\hilb(I(\ell),j) = \binom{j+2}{2}$ for $0\leq j<d$ and $\hilb(I(\ell),d) = \binom{d+2}{2}-4$.
\end{Lem}

\begin{proof}
We start as above with the products of linear forms $L_1 =\prod_{j=0}^{d-1}(x-jz)$ and $L_2 =\prod_{j=0}^{d-1}(y-jz) $ and denote by $\Gamma$ the complete intersection $\V(L_1)\cap \V(L_2)$. Split $\Gamma$ into
\[
\Gamma_2 = \{(x:y:1)\colon x+y=2\}\cup\bigcup_{j=1}^{d-4}\{(x:y:1)\colon x+y=d+j\}
\]
and $\Gamma_1=\Gamma\setminus\Gamma_2$. Then the space of forms of degree $d$ vanishing on $\Gamma_1$ has dimension $3$. Let $p$ be the uniquely determined form of degree $d$ such that $L_1,L_2,p$ is a basis of this space. By Cayley-Bacharach, we know that there is a unique relation among the point evaluations $\{\ev_x\in \R[\ul{x}]_{d-3}^\ast \colon x\in\Gamma_2\}$, say
\[
\sum_{x\in\Gamma_2} w_x\ev_x=0.
\]
Note that by the preceding Remark \ref{Rem:CBApplication}, the coefficient of $\ev_{(1:1:1)}$ is non-zero.
Set $\Gamma_1'=\Gamma_1\cup \{(1:1:1)\}$ and $\Gamma_2'=\Gamma_2'\setminus\{(1:1:1)\}$. Then the point evaluations $\{\ev_x\in\R[\ul{x}]_{d-3}^\ast\colon x\in\Gamma_2'\}$ are linearly independent and span a hyperplane $H$ in $\R[\ul{x}]_{d-3}^\ast$. So there is a unique form $q$ of degree $d-3$ vanishing on $\Gamma_2'$, namely the one vanishing on all of $\Gamma_2$, i.e.~$q=(x+y-2z)\prod_{j=1}^{d-4}(x+y-(d+j)z)$.

We will now perturb the point $(1:1:1)$ along the line $x+y=2$, see Figure \ref{fig:perturbation} for a visualisation in case $d=9$: Let $v_t := (t,2-t)$. Of course, $q(v_t)=0$ for every $t\in\R$, i.e.~the point evaluation $\ev_{v_t}\in\R[\ul{x}]_{d-3}^\ast$ lies in the hyperplane spanned by the point evaluations at $\Gamma_2'$; write
\[
\ev_{v_t} = \sum_{x\in\Gamma_2'} \alpha_x(t)\ev_x,
\]
where the coefficients $\alpha_x(t)$ are rational functions of the parameter $t$.

Suppose there is a point $P\in\Gamma_2'$ such that $\alpha_P(t)=0$ for all $t\in\R$. Then $\ev_{v_t}\in\lspan(\ev_v\colon v\in\Gamma_2'\setminus\{P\})$. Dually this means, that there is a form $f_P$ of degree $d-3$, uniquely determined modulo $q$, such that
$f_P(P)=1$,
$f_P(v)=0$ for all $v\in\Gamma_2'\setminus\{P\}$ and consequently
$f_P(v_t) = 0$ for all $t\in\R$.
Such a form cannot exist: Since $v_t$ ranges over the whole line defined by $x+y=2$, the form $f_P$ vanishes identically on this line; so we can factor it out. Furthermore, $f_P$ vanishes identically on every diagonal defining $\Gamma_2$ to the left of $P$, i.e.~$f_P(x,j-x)=0$ for all $d<j<P_1+P_2$ because it has too many zeros on these lines from $\Gamma_2'$.

Now $\Gamma_2'\cap\{x+y=P_1+P_2\}$ consists of $2d-1-P_1-P_2$ many points. We have already established $P_1+P_2-d$ linear factors of $f_P$, so the remaining cofactor has degree $2d-P_1-P_2-3$. Therefore, $f_P$ vanishes identically on this line, which is a contradiction because it contains $P$.

\begin{figure}[h]
\begin{center}
\begin{tikzpicture}
\filldraw (0,0) circle(2pt); \filldraw (0,2) circle(2pt); \filldraw (0,3) circle(2pt); \filldraw (0,4) circle(2pt); \filldraw (0,5) circle(2pt); \filldraw (0,6) circle(2pt); \filldraw (0,7) circle(2pt); \filldraw (0,8) circle(2pt); \filldraw (2,0) circle(2pt); \filldraw (2,2) circle(2pt); \filldraw (2,3) circle(2pt); \filldraw[red] (2,4) circle(2pt); \filldraw (2,5) circle(2pt); \filldraw (2,6) circle(2pt); \filldraw (2,7) circle(2pt); \filldraw (2,8) circle(2pt); \filldraw (3,0) circle(2pt); \filldraw (3,2) circle(2pt); \filldraw (3,3) circle(2pt); \filldraw (3,4) circle(2pt); \filldraw (3,5) circle(2pt); \filldraw (3,6) circle(2pt); \filldraw (3,7) circle(2pt); \filldraw (3,8) circle(2pt); \filldraw (4,0) circle(2pt); \filldraw[red] (4,2) circle(2pt); \filldraw (4,3) circle(2pt); \filldraw (4,4) circle(2pt); \filldraw (4,5) circle(2pt); \filldraw[red] (4,6) circle(2pt); \filldraw (4,7) circle(2pt); \filldraw (4,8) circle(2pt); \filldraw (5,0) circle(2pt); \filldraw (5,2) circle(2pt); \filldraw (5,3) circle(2pt); \filldraw (5,4) circle(2pt); \filldraw (5,5) circle(2pt); \filldraw (5,6) circle(2pt); \filldraw (5,7) circle(2pt); \filldraw (5,8) circle(2pt); \filldraw (6,0) circle(2pt); \filldraw (6,2) circle(2pt); \filldraw (6,3) circle(2pt); \filldraw[red] (6,4) circle(2pt); \filldraw (6,5) circle(2pt); \filldraw (6,6) circle(2pt); \filldraw (6,7) circle(2pt); \filldraw (6,8) circle(2pt); \filldraw (7,0) circle(2pt); \filldraw (7,2) circle(2pt); \filldraw (7,3) circle(2pt); \filldraw (7,4) circle(2pt); \filldraw (7,5) circle(2pt); \filldraw (7,6) circle(2pt); \filldraw (7,7) circle(2pt); \filldraw (7,8) circle(2pt); \filldraw (8,0) circle(2pt); \filldraw (8,2) circle(2pt); \filldraw (8,3) circle(2pt); \filldraw (8,4) circle(2pt); \filldraw (8,5) circle(2pt); \filldraw (8,6) circle(2pt); \filldraw (8,7) circle(2pt); \filldraw (8,8) circle(2pt);
\filldraw (3/4,0) circle(2pt); \filldraw (3/4,5/4) circle(2pt); \filldraw (3/4,2) circle(2pt); \filldraw (3/4,3) circle(2pt); \filldraw (3/4,4) circle(2pt); \filldraw (3/4,5) circle(2pt); \filldraw (3/4,6) circle(2pt); \filldraw (3/4,7) circle(2pt); \filldraw (3/4,8) circle(2pt);
\filldraw (0,5/4) circle(2pt); \filldraw (2,5/4) circle(2pt); \filldraw (3,5/4) circle(2pt); \filldraw (4,5/4) circle(2pt); \filldraw (5,5/4) circle(2pt); \filldraw (6,5/4) circle(2pt); \filldraw (7,5/4) circle(2pt); \filldraw (8,5/4) circle(2pt);
\filldraw[black!30!white] (0,1) circle(2pt); \filldraw[black!30!white] (1,1) circle(2pt); \filldraw[black!30!white] (2,1) circle(2pt); \filldraw[black!30!white] (3,1) circle(2pt); \filldraw[black!30!white] (4,1) circle(2pt); \filldraw[black!30!white] (5,1) circle(2pt); \filldraw[black!30!white] (6,1) circle(2pt); \filldraw[black!30!white] (7,1) circle(2pt); \filldraw[black!30!white] (8,1) circle(2pt);
\filldraw[black!30!white] (1,0) circle(2pt); \filldraw[black!30!white] (1,2) circle(2pt); \filldraw[black!30!white] (1,3) circle(2pt); \filldraw[black!30!white] (1,4) circle(2pt); \filldraw[black!30!white] (1,5) circle(2pt); \filldraw[black!30!white] (1,6) circle(2pt); \filldraw[black!30!white] (1,7) circle(2pt); \filldraw[black!30!white] (1,8) circle(2pt);
\draw (-1,3) -- (3,-1);
\draw (1,9) -- (9,1);
\draw (2,9) -- (9,2);
\draw (3,9) -- (9,3);
\draw (4,9) -- (9,4);
\draw (5,9) -- (9,5);
\draw[black!30!white] (1,-0.2) -- (1,8.2);
\draw[black!30!white] (-0.2,1) -- (8.2,1);
\draw[rotate=45,black!30!white] (5.65685,0) ellipse (5.65685cm and 1.46059cm);
\end{tikzpicture}
\end{center}
\caption{A picture of the perturbation for general $d\geq 4$ shown for the first critical case $d=9$: The black points and the four red points are the perturbed point configuration for which our construction works. The four red points are the additional points through which the grey ellipse goes.}
\label{fig:perturbation}
\end{figure}
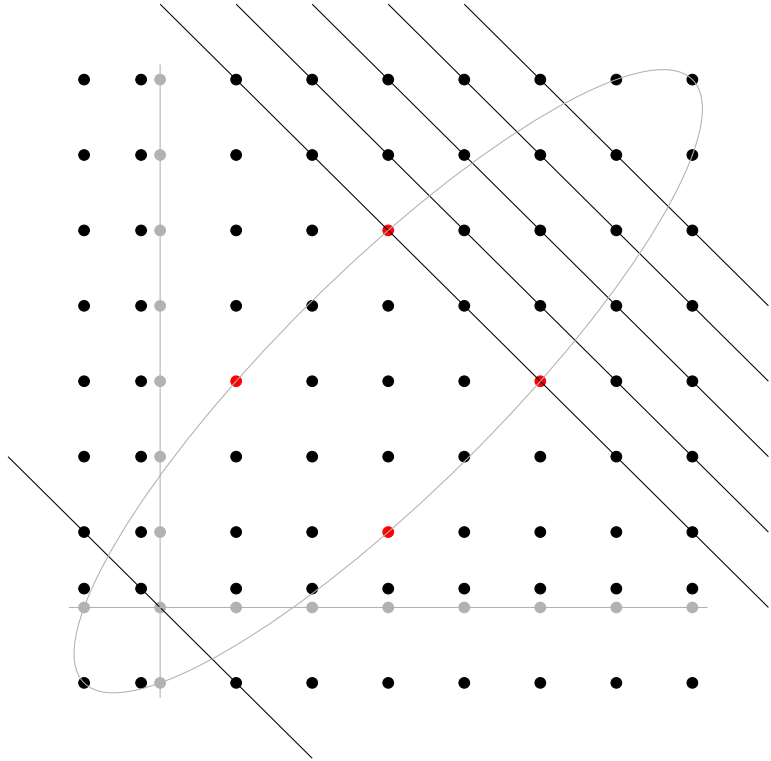

So there is an $\epsilon>0$ such that for all $t\in (1-\epsilon,1)$, all coefficients of the linear relation
\[
\ev_{v_t} = \sum_{v\in\Gamma_2'}\alpha_v(t)\ev_v
\]
are non-zero. Pick a $t_0$ in this interval and consider the totally real complete intersection $\Gamma=\V(L_1')\cap\V(L_2')$ for $L_1 = x(x-t_0z)\prod_{j=2}^{d-1}(x-jz)$ and $L_2 = y(y-(2-t_0)z)\prod_{j=2}^{d-1}(x-jz)$ and argue as above: we split the points into $\Gamma_1$ and $\Gamma_2$, where $\Gamma_2$ is the same union of diagonals as above. The Theorem of Cayley-Bacharach then implies the existence of a form of degree $d$ vanishing on $\Gamma_1$ and not identically on $\Gamma$. In fact, by Remark \ref{Rem:CBApplication}, this form does not vanish in any point of $\Gamma_2$, so we can now complete the proof as in Lemma \ref{Lem:ExtremeRaysC4}.
\end{proof}

\begin{Rem}
In particular, the union of all extreme rays of $\Sigma_{2d}^\vee$ need not be closed, e.g.~for $d=9$, extremality fails in our original construction but a perturbation gives an extreme ray. 
\end{Rem}

\begin{Thm}\label{Thm:ZarClExtRays}
For any $d\geq 4$, the Zariski closure of the set of extreme rays of $\Sigma_{2d}^\vee$ is the variety of Hankel matrices of corank at least $4$. It is irreducible, has codimension $10$, and degree $\prod_{\alpha = 0}^3 \binom{N+\alpha}{4-\alpha}/\binom{2\alpha +1}{\alpha}$, where $N = \binom{d+2}{2}$.
\end{Thm}

\begin{proof}
We have shown in the proof of Lemma \ref{Lem:Corank4dim} that the quasi-projective variety $\gor(T)$ of all Gorenstein ideals with Hilbert function $T(j)=\binom{j+2}{2}$ for $0\leq j<d$ and $T(d)=\binom{d+2}{2}-4$ is dense in $X_{-4}$. It is also smooth, cf.~Theorem \ref{Thm:GorTTangentSpace} or Iarrobino-Kanev \cite[Theorem 4.21]{IarKanMR1735271}. We have shown in Lemma \ref{Lem:ExtremeRaysC4} that there is an extreme ray $\R_+\ell_0$ of $\Sigma_{2d}^\vee$ with $I(\ell_0)\in\gor(T)$. We will now show that every linear functional in an open neighbourhood of $\ell_0$ in $\gor(T)$ spans an extreme ray of $\Sigma_{2d}^\vee$. Since $I(\ell)\in\gor(T)$ implies that the corank of the Hankel matrix $B_\ell$ is $4$, there is an open neighbourhood of $\ell_0$ such that $B_\ell$ is positive semi-definite for all $\ell$ in this neighbourhood, because the eigenvalues of a symmetric matrix depend continuously on its entries.
Therefore, a linear functional $\ell$ in this neighbourhood spans an extreme ray of $\Sigma_{2d}^\vee$ if and only if $I(\ell)_d$ generates a hyperplane in $\R[\ul{x}]_{2d}$, i.e.~$\langle I(\ell)_d\rangle_{2d}=I(\ell)_{2d}$.
By Gauss' algorithm (column echelon form), we can write a basis $(b_1,b_2,b_3,b_4)$ of the kernel of $B_\ell$ in terms of rational functions in the entries of $B_\ell$. We consider the linear map
\[
\R[\ul{x}]_d^4\to \R[\ul{x}]_{2d}, (f_1,f_2,f_3,f_4)\mapsto f_1b_1+f_2b_2+f_3b_3+f_4b_4.
\]
The rank of this map is at least $\binom{2d+2}{2}-1$ (i.e.~the image is a hyperplane) because $\ell\in\gor(T)$. The image is a hyperplane for $\ell=\ell_0$. So the same is true for every $\ell$ in a neighbourhood of $\ell_0$ in $\gor(T)$, which shows that these $\ell$ are extreme rays of $\Sigma_{2d}^\vee$.
\end{proof}

\begin{Rem}
In the proof of the above Theorem, we see that if $T$ is a Hilbert function occuring for a Gorenstein ideal corresponding to an extreme ray of $\Sigma_{2d}^\vee$, then there is an open subset of extreme rays in a connected component of $\gor(T)(\R)$ because $\gor(T)$ is smooth. As we remarked above, it might not be the entire connected component.
\end{Rem}

In fact, this gives an interesting connection to irreducible components of the algebraic boundary of the cone $\Sigma_{2d}$ of sums of squares, the Zariski closure of its boundary in the Euclidean topology.
\begin{Thm}\label{Thm:algboundandgors}
Let $X\subset\partial_a \Sigma_{2d}$ be an irreducible component. Then its dual projective variety $X^\ast$ is a subvariety of the Zariski closure of the union of extreme rays of $\Sigma_{2d}^\vee$, i.e.~the variety of Hankel matrices of corank $\geq 4$. Moreover, there is a Hilbert function $T$ such that $\gor(T)$ is Zariski dense in $X^\ast$. 
\end{Thm}

\begin{proof}
We rely on the results of \cite{SinnAlgBound} for the proof: By \cite[Proposition 3.1]{SinnAlgBound}, the projective dual variety of $X$ is contained in $\exr_a(\Sigma_{2d}^\vee)$, the Zariski closure of the union of extreme rays of $\Sigma_{2d}^\vee$ and $X^\ast\cap \exr(\Sigma_{2d}^\vee)$ is Zariski dense in $X^\ast$. So let $\ell$ be an extreme ray of $\Sigma_{2d}^\vee$ and a general point of $X^\ast$. Let $T_\ell$ be the Hilbert function of the corresponding Gorenstein ideal $I(\ell)$. Since $\gor(T_\ell)$ is smooth and every point in a neighbourhood of $\ell$ in $\gor(T_\ell)$ is also an extreme ray of $\Sigma_{2d}^\vee$, the quasiprojective variety $\gor(T_\ell)$ is Zariski dense in $X^\ast$. Indeed, the variety $\gor(T)$ is irreducible for any permissible Hilbert function $T$ and the irreducible variety $X^\ast$ is the union of some of these varieties. So one of them must be Zariski dense in $X^\ast$ and for a general $\ell\in X^\ast$, the Hilbert function of the corresponding Gorenstein ideal $I(\ell)$ identifies this variety $\gor(T)$.
\end{proof}

Our construction of an extreme ray of maximal rank also gives a base-point free special linear system with a totally real representative on a smooth curve of degree $d\geq 4$, which might be interesting in itself.
\begin{Prop}
Let $d\geq 4$. There is a smooth real curve $X\subset\P^2$ of degree $d$ and an effective divisor $D$ of degree $g = \binom{d-1}{2}$ supported on $X(\R)$ such that $|D|$ has dimension $1$ and is base-point free.
\end{Prop}

\begin{proof}
Start with a complete intersection $\V(L_1)\cap\V(L_2)$ of products of $d$ linear forms and a choice of $\binom{d-1}{2}$ points $\Gamma_2\subset\V(L_1)\cap\V(L_2)$ such that there is a unique curve of degree $d-3$ passing through these points. Moreover, assume that all coefficients in the linear relation among the point evaluations $\{\ev_v\in\R[\ul{x}]_{d-3}^\ast\colon v\in\Gamma_2\}$ are non-zero. This situation is established in the proof of Lemma \ref{Lem:ExRayC4Per}. By Bertini's Theorem \cite[Theorem 6.2.11]{BelCarMR2549804} or \cite[Th\'eor\`eme 6.6.2]{JouMR725671}, there is a smooth curve $\V(f)$ of degree $d$ passing through $\Gamma_2$ such that $f$ is a small perturbation of $L_1$; more precisely, we want $\Gamma = \V(f)\cap \V(L_2)$ to be a totally real transversal intersection. Then the complete linear system $|\Gamma_2|\subset \divi(\V(f))$ is cut out by forms of degree $d$ through $\Gamma\setminus\Gamma_2$, i.e.$|\Gamma_2|$ is the set of all effective divisors in
\[
\{C.\V(f)-(\Gamma-\Gamma_2)\colon C\subset\P^2\text{ of degree }d\},
\]
cf.~Eisenbud-Green-Harris \cite[Corollary 5 (to Brill-Noether's Restsatz)]{EisGreHar}. We have argued in Remark \ref{Rem:CBApplication} that this linear system is base-point free. We compute its dimension with the help of the Cayley-Bacharach Theorem, more precisely \cite[Corollary 6]{EisGreHar}:
\[
1 =|\Gamma_2|-( \ell( (d-3)H) - \ell( (d-3)H-\Gamma_2) )= g-( g-\ell( (d-3)H-\Gamma_2)),
\]
where $H\subset\P^2$ is a line and $\ell(D)$ is the dimension of the Riemann-Roch space of the divisor $D$.
This implies
\[
\ell(\Gamma_2) = \deg(\Gamma_2) +1 - g +\ell( (d-3)H-\Gamma_2) = 2.
\]
\end{proof}

\begin{Rem}
Conversely, given such a linear system on a smooth curve $X\subset\P^2$, we can apply the construction in the proof of Lemma \ref{Lem:ExtremeRaysC4} to construct an extreme ray of $\Sigma_{2d}^\vee$ of maximal rank, at least if there is a totally real transveral intersection $C\cap X$ with $C.X-D\geq0$. The fact, that the linear system has dimension $1$ gives the unique linear relation among the point evaluations at $C.X-D$ on forms of degree $d$. Extremality then follows from the fact that $|D|$ is base-point free by the count of dimensions as in the proof of Lemma \ref{Lem:ExtremeRaysC4}.
\end{Rem}

\section{The case $d=5$ or Ternary Decics.}\label{sec:Dextics}
For $d=3$, a complete characterization of extreme rays of $\Sigma_6^\vee$ was given by Blekherman in \cite{BleNP}. It led to a complete description of the algebraic boundary of the sums of squares cone $\Sigma_6$ by Blekherman, Hauenstein, Ottem, Ranestad and Sturmfels, cf.~\cite{BleHauOttRanStuMR2999301}.

For $d=4$, there are only two possible ranks ($>1$) for extreme rays of $\Sigma_8^\vee$, namely $10$ and $11$; in particular, we know how to construct one of each rank. For rank $10$, we use a complete intersection of a cubic and a quartic and the unique linear relation among the corresponding point evaluations on quartics to construct an extreme ray as above, see also \cite{BlePositiveGorensteinIdeals}. For rank $11$, which is the maximal rank, we use the construction from section \ref{sec:ExtRays}. The Hilbert functions of the extreme rays constructed in this way are
\begin{eqnarray*}
T_{10} & = & (1,3,6,9,10,9,6,3,1) \text{ and}\\
T_{11} & = & (1,3,6,10,11,10,6,3,1).
\end{eqnarray*}
It is possible to prove, similarly to the cases below, that both these ranks give rise to irreducible components of the algebraic boundary of $\Sigma_8$ by projective duality.

So the first new case from this point of view is $d=5$: In fact, we can construct an extreme ray of $\Sigma_{10}^\vee$ of every rank in the interval $\{13,\ldots,17\}$ between the lower and upper bound using the Cayley-Bacharach Theorem. Moreover, using the results of \cite{SinnAlgBound}, we can prove by projective duality that there is an irreducible component of the algebraic boundary of $\Sigma_{10}$ for every one of these ranks; in particular, $\partial_a\Sigma_{10}$ has at least $6$ irreducible components. In the following propositions in this section, we will prove the following theorem.
\begin{Thm}\label{Thm:d5}
For every $r\in\{13,\ldots,17\}$, there is an extreme ray $\R_+\ell_r$ of $\Sigma_{10}^\vee$ such that the rank of the Hankel matrix $B_{\ell_r}$ is $r$. The Hilbert function $T_r$ of $I(\ell_r)$ is
\begin{eqnarray*}
T_{13} & = & (1,3,6,9,12,13,12,9,6,3,1), \\
T_{14} & = & (1,3,6,10,13,14,13,10,6,3,1), \\
T_{15} & = & (1,3,6,10,14,15,14,10,6,3,1), \\
T_{16} & = & (1,3,6,10,14,16,14,10,6,3,1), \\
T_{17} & = & (1,3,6,10,15,17,15,10,6,3,1). 
\end{eqnarray*}
The dual varieties to $\gor(T_r)$ are irreducible components of the algebraic boundary of the sums of squares cone $\Sigma_{10}$ for all $r\in\{13,\ldots,17\}$. The variety $\gor(T_{17})$ is Zariski dense in the Zariski closure of the union of all extreme rays. It has dimension $55$ and degree $53300016$.
\end{Thm}

The construction given in the preceding section for extreme rays of maximal rank $\binom{d+2}{2}-4$ leads to an extreme ray $\R_+\ell$ of $\Sigma_{10}^\vee$ such that the Hilbert function of the corresponding Gorenstein ideal $I(\ell)$ is
\[
T_{17} = (1,3,6,10,15,17,15,10,6,3,1).
\]
By Theorem \ref{Thm:ZarClExtRays}, the Zariski closure of the set of extreme rays of $\Sigma_{10}^\vee$ is $\cl(\gor(T_{17}))$, a unirational variety of codimension $10$ in $\P^{65}$. So \cite[Theorem 3.8]{SinnAlgBound}, implies that its dual variety is an irreducible component of the algebraic boundary of $\Sigma_{10}$.

We now work our way up beginning with the lowest rank $13$, following the construction in Blekherman \cite{BlePositiveGorensteinIdeals}:
\begin{Prop}
There is an extreme ray $\R_+\ell$ of $\Sigma_{10}^\vee$ of rank $13$. The Hilbert function of the Gorenstein ideal $I(\ell)$ is
\[
T_{13}=(1,3,6,9,12,13,12,9,6,3,1)
\]
and the variety dual to $\cl(\gor(T_{13}))$ is an irreducible component of the algebraic boundary of $\Sigma_{10}$.
\end{Prop}

\begin{proof}
Let $L_1 = x(x-z)(x-2z)(x-3z)(x-4z)$ and $L_2 = y(y-z)(y-2z)$ and $\Gamma = \V(L_1)\cap\V(L_2)$. By construction, there is a unique linear relation among $\{\ev_v\in\R[\ul{x}]_5^\ast\colon v\in\Gamma\}$, say $\sum_{v\in\Gamma}u_v\ev_v = 0$, and all coefficients in this relation are non-zero. The linear functional
\[
\ell = \sum_{v\in\Gamma\setminus\{P\}} ev_v - \frac{u_P^2}{\sum_{v\in\Gamma\setminus\{P\}} u_v^2} \ev_P
\]
is positive semi-definite of rank $13$ for any $P\in\Gamma$ by the Cauchy-Schwarz inequality, cf.~proof of Lemma \ref{Lem:ExtremeRaysC4}. By a Hilbert function computation using \texttt{Macaulay2} \cite{Macaulay2}, we verify, that the degree $5$ part of the corresponding Gorenstein ideal $I(\ell)$ generates a hyperplane in degree $10$. To prove that the dual variety to $\gor(T_{13})$ is an irreducible component of $\partial_a \Sigma_{10}$, we use \cite[Theorem 3.8]{SinnAlgBound}. The condition given there is equivalent to
\[
(T_{\ell} \gor(T_{13}))^\perp = (I(\ell)_5)^2
\]
because the face of $\Sigma_{10}$ supported by $\ell$ is the set of sums of squares of polynomials in $I(\ell)_5$, which spans the vector space $(I(\ell)_5)^2$. By the description of the tangent space to $\gor(T_{13})$ at $\ell$ (cf.~Theorem \ref{Thm:GorTTangentSpace}), this is equivalent to
\[
(I(\ell)^2)_{10} = (I(\ell)_5)^2,
\]
which we also check using \texttt{Macaulay2} \cite{Macaulay2}.
\end{proof}

\begin{Prop}
There is an extreme ray $\R_+\ell$ of $\Sigma_{10}^\vee$ of rank $14$. The Hilbert function of the Gorenstein ideal $I(\ell)$ is
\[
T_{14}=(1,3,6,10,13,14,13,10,6,3,1)
\]
and the variety dual to $\cl(\gor(T_{14}))$ is an irreducible component of the algebraic boundary of $\Sigma_{10}$.
\end{Prop}

\begin{proof}
In this case, take $L_1 = x(x-z)(x-2z)(x-3z)$ and $L_2 = y(y-z)(y-2z)(y-3z)$ and set $\Gamma = \V(L_1)\cap \V(L_2)$. There is a unique linear relation among $\{\ev_v\in\R[\ul{x}]_5^\ast \colon v\in\Gamma\}$, say $\sum_{v\in\Gamma} u_v\ev_v = 0$, and all its coefficients are non-zero. As above, the linear functional
\[
\ell = \sum_{v\in\Gamma\setminus\{P\}} ev_v - \frac{u_P^2}{\sum_{v\in\Gamma\setminus\{P\}} u_v^2} \ev_P
\]
is positive semi-definite of rank $14$ for any $P\in\Gamma$. Again, using \texttt{Macaulay2} \cite{Macaulay2}, we verify, that the degree $5$ part of the corresponding Gorenstein ideal $I(\ell)$ generates a hyperplane in degree $10$ and that
\[
(I(\ell)^2)_{10} = (I(\ell)_5)^2.
\]
\end{proof}

\begin{Prop}
There is an extreme ray $\R_+\ell$ of $\Sigma_{10}^\vee$ of rank $15$. The Hilbert function of the Gorenstein ideal $I(\ell)$ is
\[
T_{15}=(1,3,6,10,14,15,14,10,6,3,1)
\]
and the variety dual to $\cl(\gor(T_{15}))$ is an irreducible component of the algebraic boundary of $\Sigma_{10}$.
\end{Prop}

\begin{proof}
In this case, we start with a complete intersection of a quartic and a quintic, $L_1 = x(x-z)(x-2z)(x-3z)(x-4z)$, $L_2 = y (y-z)(y-2z)(y-3z)$ and $\Gamma = \V(L_1)\cap \V(L_2)$. Choose $\Gamma_2 = \{(0:2:1),(1:1:1),(2:0:1)\}$ and set $\Gamma_1 = \Gamma\setminus\Gamma_2$. By Cayley-Bacharach, there is a unique linear relation among the $17$ points of $\Gamma_1$. Using \texttt{Macaulay2} \cite{Macaulay2}, we complete the proof as above.
\end{proof}

\begin{Prop}
There is an extreme ray $\R_+\ell$ of $\Sigma_{10}^\vee$ of rank $16$. The Hilbert function of the Gorenstein ideal $I(\ell)$ is
\[
T_{16}=(1,3,6,10,14,16,14,10,6,3,1)
\]
and the variety dual to $\cl(\gor(T_{16}))$ is an irreducible component of the algebraic boundary of $\Sigma_{10}$.
\end{Prop}

\begin{proof}
Choose $L_1 = x(x-z)(x-2z)(x-3z)(x-4z)$, $L_2 = y (y-z)(y-2z)(y-3z)(y-4z)$ and $\Gamma = \V(L_1)\cap \V(L_2)$. This time, $\Gamma_2 = \{(0:2:1),(1:1:1),(2:0:1),(1:4:1),(2:3:1),(3:2:1),(4:1:1)\}$ and $\Gamma_1 = \Gamma\setminus\Gamma_2$ do the job: Cayley-Bacharach gives a unique linear relation among the $18$ points of $\Gamma_1$. Using \texttt{Macaulay2} \cite{Macaulay2}, we complete the proof as above.
\end{proof}

For general $d>5$, our constructive method using the Cayley-Bacharach Theorem cannot construct an extreme ray of every rank in the interval $\{3d-2,\ldots,\binom{d+2}{2}\}$ given by the lower and upper bound. The first failure occurs for $d=6$ and rank $17\in\{16,\ldots,24\}$. In fact, $\Sigma_{12}^\vee$ does not have an extreme ray of rank $17$, as we will see below, cf.~Lemma \ref{Lem:Rank17}. Let us first argue why we cannot construct an extreme ray of rank $17$ of $\Sigma_{12}^\ast$.
\begin{Rem}\label{Rem:d6r17}
Our construction starts with a totally real intersection of two curves $X_1$ and $X_2$ with $\deg(X_1)+\deg(X_2)\geq d+3$; we then need $19$ intersection points such that the corresponding point evaluations on forms of degree $6$ satisfy a unique linear relation in which all coefficients are non-zero. This configuration would lead to a positive linear functional such that the Hankel matrix has the desired rank $17$ (of course we would still need to prove extremality). We will see that this is not possible:

The following tuples are permissible choices for the degrees of the curves $(3,6)$, $(4,5)$, $(4,6)$, $(5,5)$, $(5,6)$ and $(6,6)$. 
For $(\deg(X_1),\deg(X_2)) = (3,6)$, the transversal intersection has only $18$ points. In the case $(4,5)$, there is a unique linear relation among point evaluations at the $20$ intersection points such that all coefficients are non-zero; in particular, whatever point we remove, the remaining $19$ point evaluations are linearly independent on forms of degree $6$. In the cases $(4,6)$, $(5,5)$ and $(5,6)$, we cannot have the desired number of points on a curve of dual degree $s-d$: For example, in order to apply the duality of the Cayley-Bacharach Theorem to the $24$ intersection points in the case $(4,6)$, we would need to have $5$ of the intersection points on a line, which intersects the quartic in only $4$ points. The last case $(6,6)$ is more subtle: We would like to find exactly $17$ intersection points on a cubic, which is impossible, because there is a unique linear relation among the corresponding point evaluations on forms of degree $6$ on the complete intersection of a cubic and a sextic, cf.~Eisenbud-Green-Harris \cite[CB4]{EisGreHar}.
\end{Rem}

This is not a defect of our construction in this case. In fact, there are no extreme rays of $\Sigma_{12}^\vee$ of rank $17$.
\begin{Lem}\label{Lem:Rank17}
There is no Gorenstein ideal $I\subset \C[x,y,z]$ with socle in degree $12$ such that $\hilb(I,6) = 17$ and $I_6$ is maximal with respect to inclusion among $J_6$, where $J$ runs over all Gorenstein ideals with socle in degree $12$.
\end{Lem}

In the proof of this lemma, we will use the following theorem multiple times.
\begin{Rem}
The complete intersection of three ternary forms of degree $d_1$, $d_2$, and $d_3$ respectively is a Gorenstein ideal in $\C[x,y,z]$ with socle in degree $d_1+d_2+d_3-3$, see \cite[Theorem CB8]{EisGreHar}. The socle degree follows using an elementary count of dimensions and the fact that the generators must be relatively prime.
\end{Rem}

\begin{proof}
We exclude possibilities arguing by the lowest degree $k$ of a generator of $I$. First note that maximality of $I_6$ implies that $\langle I_6 \rangle_{12} = I_{12}$ and $\V(I_6)=\emptyset$. In particular, we can always choose a complete intersection of three forms in $I_6$, one of which can be chosen to be a suitable multiple of a generator of minimal degree of $I$. Let $k$ be the minimal degree of a generator of $I$. The ideal $I$ cannot contain a quadric generator, because the linear function defining $I$ would then be supported on $12$ points by the apolarity lemma (see \cite[Chapter I]{IarKanMR1735271}), which implies $\hilb(I,6) \leq 12$. In case $k=3$, the Gorenstein ideal is actually generated by a cubic and two sextics that are a complete intersection, see Stanley \cite{Stanley}, so $\hilb(I,6) = 16 = 28 - (10 +2)$. The case $k=6$ is also easily excluded because Hilbert functions of Gorenstein ideals in $\C[x,y,z]$ are unimodal by \cite[Theorem 4.2]{Stanley}, which implies in this case that $\hilb(I,5)\leq \hilb(I,6) = 17$, or equivalently $\dim(I_5) \geq 4$. This leaves the two cases $k=5$ and $k=4$: 

Suppose $k=5$, then the Hilbert functions of Gorenstein ideals with socle in degree $12$ and order $5$ are in 1-1 correspondence with self-complimentary partitions of $10\times 4$ blocks, cf.~\cite[Proposition 3.9]{Die}. By unimodality of Hilbert functions, we have $\dim(I_5)\geq 4$. This determines the first three rows of the blocks, so we can choose two more generators of degree $\leq 6$. Since a block of degree $5$ forces a relation in degree $6$ by the self-complimentarity of the partition, the $4$ generators of degree $5$ generate a $4\cdot 3 - 3 = 9$-dimensional subspace of $\C[x,y,z]_6$. The only two possible degrees for the other two generators to achieve $\dim(I_6) = 28-17 = 11$ are therefore one more generator of degree $5$ and one generator of degree $7$ or two generators in degree $6$. In the first case, the degrees of generators are $q = (5,5,5,5,5,7,7,9,9,9,9)$ and the corresponding relation degrees are $p = (10,10,10,10,10,8,8,6,6,6,6)$. Since there is no generator of degree $6$ and $\V(I_6) = \emptyset$, we find a complete intersection of three forms of degree $5$ contained in $I$. These generate a Gorenstein ideal with socle in degree $5+5+5-3 = 12$. This is impossible, becaues $I$ contains further generators. In the second case, the degrees of generators are $q = (5,5,5,5,6,6,8,8,9,9,9)$ and the corresponding relation degrees are $p = (10,10,10,10,9,9,7,7,6,6,6)$. The quasiprojective variety $\gor(T)$, where $T$ is given by these generator and relation degrees, contains a dense subset of Gorenstein ideals generated by polynomials of degree $q_{\rm min} = (5,5,5,5,8,8,9)$ by \cite[Section 3.3]{Die}. So the same argument as in the first case excludes this possibility, too.

So now we are left with the case $k=4$: In this case, Hilbert functions of Gorenstein ideals with socle in degree $12$ and order $4$ correspond to self-complimentary partitions of $8\times 6$ blocks. If $\dim(I_4) =2$, then these would generate a $2\cdot 6 = 12$-dimensional subspace in $\C[x,y,z]_6$ because they correspond to relations in degree $7$. So $\dim(I_4) = 1$. We can choose $4$ more degrees of generators $\leq 6$. A generator of degree $5$ comes with a relation in degree $5$ and therefore, the only two possible choices of degrees for the generators with $\hilb(I,6) = 17$ are one generator of degree $5$ and three generators of degree $6$ or two generators of degree $5$ and one generator of degree $6$. Let's first consider $q = (4,5,5,6,7,7,8,9,9)$, which corresponds to relation degrees $p = (11,10,10,9,8,8,7,6,6)$. Again, $\gor(T)$, where $T$ is given by these generator and relation degrees, contains a dense subset of Gorenstein ideals generated by polynomials of degree $q_{\rm min} = (4,5,5,7,9)$. There is no generator of degree $6$ anymore, so $\V(I_6) = \emptyset$ implies that we find a complete intersection of a quartic and two quintics, which generate a Gorenstein ideal with socle in degree $4+5+5-3=11$, which is impossible. The last remaining case is $q = (4,5,6,6,6,8,8,8,9)$ with corresponding relation degrees $p = (11,10,9,9,9,7,7,7,6)$. Here, $\gor(T)$ contains a dense subset of Gorenstein ideals with generators of degree $q_{\rm min} = (4,5,6,6,8,8,8)$ and correspondingly $p_{\rm min} = (11,10,9,9,7,7,7)$. The assumption $\V(I_6) = \emptyset$ implies only that we can find a complete intersection of a quartic and two sextics. They generate a Gorenstein ideal with socle in degree $4+6+6-1 = 13$. This Gorenstein ideal has Hilbert function $3$ in degree $12$ because Hilbert functions of Gorenstein ideals are symmetric, cf.~Remark \ref{Rem:SymHilbFunc}. This is impossible because this complete intersection together with the other $4$ generators would then fill $\C[x,y,z]_{12}$. This concludes the case study.
\end{proof}

We can use similar ideas as above for $d=5$ to show that $\Sigma_{12}^\vee$ has extreme rays of ranks $18,\ldots,23$. We start with the complete intersection of a sextic with a quartic, quintic, or sextic in $24$, $30$, and $36$ points respectively and remove the desired number of points on the appropriate number of lines to get a unique linear relation among the point evaluations on sextics at the remaining points. We can then use \texttt{Macaulay2} \cite{Macaulay2} to check extremality as before. Again, for all these ranks, the projective dual variety to $\gor(T_r)$ will be an irreducible component of the algebraic boundary of $\Sigma_{12}$.

\textbf{Acknowledgements.} This paper grew out of the PhD thesis of the second author, who wants to thank his adviser Claus Scheiderer for his support, ideas and input. 
The first author was partially supported by Alfred P. Sloan research fellowship and NSF DMS CAREER award. The second author was supported by the Studienstiftung des deutschen Volkes and the National Institute of Mathematical Sciences, Daejeon, Korea, during the Summer 2014 Thematic Program on {\em Applied Algebraic Geometry}.

\end{document}